\documentclass[11pt]{amsart}
\usepackage{a4wide}
\usepackage{booktabs}
\usepackage{url}

\newtheorem{thm}{Theorem}[section]
\newtheorem{cor}[thm]{Corollary}

\numberwithin{equation}{section}

\renewcommand{\)}{\rangle}
\newcommand{\para}{$\wp$}
\DeclareMathOperator{\Sym}{Sym}

\newcommand{\Magma}{\textsc{Magma}}

\renewcommand\today{\number\day\space\ifcase\month\or
  January\or February\or March\or April\or May\or June\or
  July\or August\or September\or October\or November\or December\fi
  \space\number\year}

\author{D. E. Taylor}
\title[Reflection subgroups of primitive reflection groups]%
{Simple extensions of reflection subgroups of primitive complex reflection groups}
\address{School of Mathematics and Statistics\\
The University of Sydney\\
Australia 2006}
\email{donald.taylor@sydney.edu.au}
\date{\today}

\begin{document}
\begin{abstract}
If $G$ is a finite primitive complex reflection group, all reflection
subgroups of $G$ and their inclusions are determined up to conjugacy. As a
consequence, it is shown that if the rank of $G$ is $n$ and if $G$ can be
generated by $n$ reflections, then for every set $R$ of $n$ reflections which
generate $G$, every subset of $R$ generates a parabolic subgroup of $G$.
\end{abstract}


\maketitle

\section{Introduction}
The finite irreducible complex reflection groups were classified by Shephard
and Todd \cite{shephard-todd:1954} in 1954. If $G$ is a primitive complex
reflection group then, as shown by Shephard and Todd, $G$ is either cyclic, a
symmetric group $\Sym(n)$ for $n\ge 5$, or one of 34 groups $G_k$, where $4
\le k\le 37$.

A \emph{reflection subgroup} of $G$ is a subgroup generated by reflections. A
\emph{parabolic subgroup} is the pointwise stabiliser of a subset $X$ of $V$.
By a fundamental theorem of Steinberg \cite{steinberg:1964} (see also
\cite[Theorem 9.44]{lehrer-taylor:2009}) a parabolic subgroup is a reflection
subgroup.

If $H$ is a reflection subgroup of $G$, the \emph{simple extensions} of $H$
are the subgroups $\(H,r\)$, where $r$ is a reflection and $r\notin H$.

If $\mathcal H$ is a conjugacy class of reflection subgroups of $G$, a
conjugacy class $\mathcal K$ is a \emph{simple extension} of $\mathcal H$ if
there exists $H\in\mathcal H$ and $K\in \mathcal K$ such that $K$ is a simple
extension of $H$.

All simple extensions of the conjugacy classes of reflection subgroups of the
imprimitive complex reflection groups $G(m,p,n)$ were determined in
\cite{taylor:2011}. The purpose of the present paper is to extend this result
to all finite complex reflection groups by describing the simple extensions
of the conjugacy classes of reflection subgroups of the groups $G_k$ ($23\le
k \le 37$).  For the groups $G_k$ of rank 2 ($4 \le k \le 22$) every element
is a reflection modulo scalars and the simple extensions can be deduced from
the results of \cite[Chapter 6]{lehrer-taylor:2009}.

The results are presented in Section \ref{sec:tab} in the form of tables. The
tables themselves were computed with the aid of the computational algebra
system \Magma\ \cite{magma:1997}.  Tables of conjugacy classes of the
reflection subgroups of the Coxeter groups of types $E_6$, $E_7$, $E_8$,
$F_4$, $H_3$ and $H_4$ can be found in \cite{douglass-etal:2011}.  Tables of
conjugacy classes of parabolic subgroups of these groups also appear in
\cite[Appendix A]{geck-pfeiffer:2000}.

Refer to \cite{lehrer-taylor:2009} and \cite{taylor:2011} for background and
terminology not otherwise explained here.

\section{Notation}
In his thesis Cohen \cite{cohen:1976} introduced a notation for primitive
complex reflection groups of rank at least 3 which extends the standard
(Cartan) notation for Coxeter groups. In this notation the complex reflection
groups which are not Coxeter groups are labelled $J_3^{(4)}$, $J_3^{(5)}$,
$K_5$, $K_6$, $L_3$, $L_4$, $M_3$, $N_4$ and $EN_4$.  In the tables which
follow we shall label the conjugacy classes of reflection subgroups using
this notation except that, as in \cite{lehrer-taylor:2009}, we use $O_4$
instead of $EN_4$.

A reflection subgroup which is the direct product of irreducible reflection
groups of types $T_1$, $T_2$,\dots,~$T_k$ will be labelled
$T_1+T_2+\cdots+T_k$ and if $T_i = T$ for all $i$ we denote the group by
$kT$.

For the imprimitive reflection subgroups which occur in the tables we use the
notation introduced in \cite[section 7.5]{lehrer-taylor:2009} rather than the
Shephard and Todd notation $G(m,p,n)$.  That is, $B_n^{(2p)}$ denotes the
group $G(2p,p,n)$ and $D_n^{(p)}$ denotes the group $G(p,p,n)$. For
consistency with the Cartan names we write $B_n$ instead of $B_n^{(2)}$ and
$D_n$ instead of $D_n^{(2)}$.  Similarly $A_{n-1}$ denotes the symmetric
group $\Sym(n)\simeq G(1,1,n)$.  However, we use $D_2^{(m)}$ rather than
$I_2(m)$ to denote the dihedral group of order $2m$.

For small values of the parameters there are isomorphisms between the groups:
$B_2\simeq D_2^{(4)}$, $A_2\simeq D_2^{(3)}$, $A_3\simeq D_3$ and $B_2 \simeq
D_2^{(4)}$.  The tables use the first named symbol for these groups. The
cyclic groups of order $2$ and $3$ are denoted by $A_1$ and $L_1$
respectively, and $L_2$ denotes the Shephard and Todd group~$G_4$.

If there is more than one conjugacy class of reflection subgroups of type $T$
we label the conjugacy classes $T.1$, $T.2$, and so on. There is no
significance to the order in which these indices occur.

Given a reflection subgroup $H$, the \emph{parabolic closure} of $H$ is the
pointwise stabiliser of the space of fixed points of $H$; it is the smallest
parabolic subgroup which contains $H$. The rank of the parabolic closure
equals is equal to the rank of $H$. For those (conjugacy classes of)
reflection subgroups $H$ whose parabolic closure is a simple extension of $H$
we place the parabolic closure first in the list of simple extensions and use
a bold font.

The conjugacy classes of parabolic subgroups are labelled with the symbol
\para.

\section{Main theorem}
\begin{thm}\label{thm:extpara}
Suppose that $H$ is a reflection subgroup of the finite primitive complex
reflection group $G$ and suppose that $K$ is a simple extension of $H$. If
$K$ is parabolic and the rank of $K$ is greater than the rank of $H$, then
$H$ is parabolic.
\end{thm}

\begin{proof}
If $G$ is the symmetric group $\Sym(n)$, it is well known---see
\cite[Corollary 3.9]{taylor:2011} for a proof---that every reflection
subgroup of $G$ is parabolic.  Thus in this case there is nothing to prove.
If the rank of $G$ is 1, the only parabolic subgroup is $G$ itself and so we
may suppose that $G = G_k$ for some $k$ such that $4\le k \le 37$.

If the rank of $G$ is 2 and $H$ is a non-parabolic reflection subgroup of
rank 1, then from \cite[Table D.1]{lehrer-taylor:2009} $G$ contains an
element of order 4 whose square generates $H$.  If $H$ has a simple extension
of rank 2 which is parabolic, then the simple extension is $G$ and so $G$ is
generated by two reflections.  The only possibilities for $G$ are $G_8$ and
$G_9$ but from \cite[Section 6.3]{lehrer-taylor:2009} neither group can be
generated by two reflections one of which is the square of a reflection of
order~4.

If the rank of $G$ is at least 3, the result follows from an inspection of
the tables in Section~\ref{sec:tab}.
\end{proof}

\begin{cor}\label{thm:prim}
If $G$ is a primitive reflection group of rank $n$ and if $R$ is a set of $n$
reflections which generate $G$, then for any subset $S$ of $R$, the subgroup
generated by $S$ is a parabolic subgroup of $G$.
\end{cor}

\section{The \Magma\ code}

In order to construct the complex reflection group $W = G_n$, where $4\le
n\le 37$, use the \Magma\ code
\begin{verbatim}
      roots, coroots, rho, W, J := ComplexRootDatum(n);
\end{verbatim}
In addition to $W$ this function returns a set of roots, a set of coroots, a
bijection \texttt{rho} from \texttt{roots} to \texttt{coroots} and a matrix
$J$ which defines a $W$-invariant positive definite hermitian form.  The
reflection \texttt{r} with root \texttt{a} and coroot \texttt{rho(a)} can be
obtained via the code
\begin{verbatim}
      r := PseudoReflection(a,rho(a));
\end{verbatim}

Given a reflection subgroup $H$ of $W$, we create a sequence \texttt{extn} of
simple extensions of $H$ (up to conjugacy):

\begin{enumerate}
  \item Let \texttt{orbreps} be a set of representatives for the orbits
      of the normaliser of $H$ in $W$ on the reflections which do not
      belong to~$H$.

  \medskip
  \item For each reflection \texttt{r} in \texttt{orbreps} construct the
      simple extension
  \begin{verbatim}
    G := sub< W | H, r >;
  \end{verbatim}
  \item If $G$ is not conjugate in $W$ to any simple extension already
      constructed, append $G$ to \texttt{extn}.
\end{enumerate}

Identification of the type of a reflection subgroup $H$ is carried out as
follows.
\begin{enumerate}
  \item Compute the list of indecomposable components $L_1$, $L_2$,
      \dots,~$L_k$ of the root system $L$ of $H$; that is, $L$ is the
      union of the $L_i$, the $L_i$ are pairwise orthogonal and the
      reflection subgroup of $H$ corresponding to $L_i$ is irreducible
      (as a reflection group).

  \medskip
  \item Compute the standard name of each indecomposable component of
      $H$. This is facilitated by the observation that the irreducible
      reflection groups $K$ which occur in the tables are uniquely
      determined by the pair of integers $(n,m)$, where $n$ is the order
      of $K$ and $m$ is the size of its line system.

  \medskip
  \item An associative array \texttt{refgroup} is used to map the
      standard name of a reflection subgroup to the actual subgroup.
  \begin{verbatim}
     refgroup := AssociativeArray(Parent(""));
  \end{verbatim}

\end{enumerate}

The full implementation of the \Magma\ code is available at
\begin{center}
\url{http://www.maths.usyd.edu.au/u/don/}
\end{center}
in the file \texttt{subsystems.m}. The function \texttt{showTable} creates
the data which is the basis for the tables in Section \ref{sec:tab}.  For
example
\begin{verbatim}
  load "subsystems.m";
  showTable(23);
\end{verbatim}
displays the data
\begin{verbatim}
  P | A1 | [ A1A1, A2, D2(5) ]
  P | A1A1 | [ H3, A1A1A1 ]
  P | A2 | [ H3 ]
  P | D2(5) | [ H3 ]
  P | H3 | []
  N | A1A1A1 | [ H3 ]
\end{verbatim}

\newpage
\section{The tables}\label{sec:tab}

\let\pcl\mathbf
\let\urg\relax

\begin{table}[h]
\caption{Reflection subgroup classes of $G_{23} = \urg H_3$}\label{tab:H3}
\centering
\begin{tabular}{cll}
\toprule
&Class&Simple extensions\\
\toprule
\para&$\urg A_1$&$\urg D_2^{(5)},\ \urg A_2,\ 2\urg A_1$\\
\midrule
\para&$2\urg A_1$&$\urg H_3,\ 3\urg A_1$ \\
\para&$\urg A_2$&$\urg H_3$ \\
\para&$\urg D_2^{(5)}$&$\urg H_3$ \\
\midrule
&$3\urg A_1$&$\pcl H_3$ \\
\bottomrule
\end{tabular}
\end{table}

\vskip 0.5cm

\begin{table}[h]
\caption{Reflection subgroup classes of $G_{24}= \urg
J_3^{(4)}$}\label{tab:J34} \centering
\begin{tabular}{cll}
\toprule
&Class&Simple extensions\\
\toprule
\para&$\urg A_1$&$\urg B_2,\ \urg A_2,\ 2\urg A_1.1,\ 2\urg A_1.2$\\
\midrule
&$2\urg A_1.1$&$\pcl B_2,\ \urg B_3.1,\
  \urg A_1+\urg B_2,\ \urg A_3.1,\ 3\urg A_1.1$\\
&$2\urg A_1.2$&$\pcl B_2,\ \urg B_3.2,\
  \urg A_1+\urg B_2,\ \urg A_3.2,\ 3\urg A_1.2$\\
\para&$\urg A_2$&$\urg J_3^{(4)},\ \urg B_3.1,\ \urg B_3.2,\
  \urg A_3.1,\ \urg A_3.2$\\
\para&$\urg B_2$&$\urg J_3^{(4)},\ \urg B_3.1,\
  \urg B_3.2,\ \urg A_1+\urg B_2$\\
\midrule
&$3\urg A_1.1$&$\urg B_3.1,\ \urg A_1+\urg B_2$\\
&$3\urg A_1.2$&$\urg B_3.2,\ \urg A_1+\urg B_2$\\
&$\urg A_1+\urg B_2$&$\pcl J_3^{(4)},\ \urg B_3.1,\
  \urg B_3.2$\\
&$\urg A_3.1$&$\pcl J_3^{(4)},\ \urg B_3.1$\\
&$\urg A_3.2$&$\pcl J_3^{(4)},\ \urg B_3.2$\\
&$\urg B_3.1$&$\pcl J_3^{(4)}$\\
&$\urg B_3.2$&$\pcl J_3^{(4)}$\\
\bottomrule
\end{tabular}
\end{table}

\vskip 0.5cm

\begin{table}[h]
\caption{Reflection subgroup classes of $G_{25}= \urg L_3$}\label{tab:L3}
\centering
\begin{tabular}{cll}
\toprule
&Class&Simple extensions\\
\toprule
\para&$\urg L_1$&$\urg L_2,\ 2\urg L_1$\\
\midrule
\para&$2\urg L_1$&$\urg L_3,\ 3\urg L_1$\\
\para&$\urg L_2$&$\urg L_3$\\
\midrule
&$3\urg L_1$&$\pcl L_3$\\
\bottomrule
\end{tabular}
\end{table}

\begin{table}
\caption{Reflection subgroup classes of $G_{26}= \urg M_3$}\label{tab:M3}
\centering
\begin{tabular}{cll}
\toprule
&Class&Simple extensions\\
\toprule
\para&$\urg L_1$&$\urg L_2,\ 2\urg L_1,\ \urg B_2^{(3)},\ \urg A_1+\urg L_1$\\
\para&$\urg A_1$&$\urg B_2^{(3)},\ \urg A_1+\urg L_1,\ \urg A_2$\\
\midrule
&$2\urg L_1$&$\pcl B_2^{(3)},\ \urg L_3,\ \urg B_2^{(3)}+\urg L_1,\
  3\urg L_1$\\
&$\urg A_2$&$\pcl B_2^{(3)},\ \urg B_3^{(3)},\ \urg D_3^{(3)},\
  \urg A_2+\urg L_1$\\
\para&$\urg L_2$&$\urg M_3,\ \urg L_3,\ \urg A_1+\urg L_2$\\
\para&$\urg B_2^{(3)}$&$\urg M_3,\ \urg B_3^{(3)},\ \urg B_2^{(3)}+\urg L_1$\\
\para&$\urg A_1+\urg L_1$&$\urg M_3,\ \urg B_3^{(3)},\
  \urg B_2^{(3)}+\urg L_1,\ \urg A_1+\urg L_2,\ \urg A_2+\urg L_1$\\
\midrule
&$\urg D_3^{(3)}$&$\urg B_3^{(3)}$\\
&$3\urg L_1$&$\urg L_3,\ \urg B_2^{(3)}+\urg L_1$\\
&$\urg B_2^{(3)}+\urg L_1$&$\pcl M_3,\ \urg B_3^{(3)}$\\
&$\urg A_2+\urg L_1$&$\pcl M_3,\ \urg B_3^{(3)},\ \urg B_2^{(3)}+\urg L_1$\\
&$\urg B_3^{(3)}$&$\pcl M_3$\\
&$\urg L_3$&$\pcl M_3$\\
&$\urg A_1+\urg L_2$&$\pcl M_3$\\
\bottomrule
\end{tabular}
\vskip 0.25cm
\end{table}

\begin{table}
\caption{Reflection subgroup classes of $G_{27}= \urg
J_3^{(5)}$}\label{tab:J35} \centering
\begin{tabular}{cll}
\toprule
&Class&Simple extensions\\
\toprule
\para&$\urg A_1$&$\urg B_2,\ \urg D_2^{(5)},\ \urg A_2.1,\ \urg A_2.2,\
  2\urg A_1.1,\ 2\urg A_1.2$\\
\midrule
&$2\urg A_1.1$&$\pcl B_2,\ \urg H_3.1,\ \urg B_3.1,\
  \urg A_1+\urg B_2,\ \urg A_3.1,\ 3\urg A_1.1$\\
&$2\urg A_1.2$&$\pcl B_2,\ \urg H_3.2,\ \urg B_3.2,\
  \urg A_1+\urg B_2,\ \urg A_3.2,\ 3\urg A_1.2$\\
\para&$\urg A_2.1$&$\urg J_3^{(5)},\ \urg H_3.2,\ \urg B_3.1,\
  \urg D_3^{(3)},\ \urg A_3.1$\\
\para&$\urg A_2.2$&$\urg J_3^{(5)},\ \urg H_3.1,\ \urg B_3.2,\
  \urg D_3^{(3)},\ \urg A_3.2$\\
\para&$\urg D_2^{(5)}$&$\urg J_3^{(5)},\ \urg H_3.1,\ \urg H_3.2$\\
\para&$\urg B_2$&$\urg J_3^{(5)},\ \urg B_3.1,\
  \urg B_3.2,\ \urg A_1+\urg B_2$\\
\midrule
&$3\urg A_1.1$&$\urg H_3.1,\ \urg B_3.1,\ \urg A_1+\urg B_2$\\
&$3\urg A_1.2$&$\urg H_3.2,\ \urg B_3.2,\ \urg A_1+\urg B_2$\\
&$\urg A_3.1$&$\pcl J_3^{(5)},\ \urg B_3.1$\\
&$\urg A_3.2$&$\pcl J_3^{(5)},\ \urg B_3.2$\\
&$\urg A_1+\urg B_2$&$\pcl J_3^{(5)},\ \urg B_3.1,\
  \urg B_3.2$\\
&$\urg H_3.1$&$\pcl J_3^{(5)}$\\
&$\urg H_3.2$&$\pcl J_3^{(5)}$\\
&$\urg B_3.1$&$\pcl J_3^{(5)}$\\
&$\urg B_3.2$&$\pcl J_3^{(5)}$\\
&$\urg D_3^{(3)}$&$\pcl J_3^{(5)}$\\
\bottomrule
\end{tabular}
\end{table}

\begin{table}
\caption{Reflection subgroup classes of $G_{28}= \urg F_4$}\label{tab:F4}
\centering
\begin{tabular}{cll}
\toprule
&Class&Simple extensions\\
\toprule
\para&$\urg A_1.1$&$\urg B_2,\ \urg A_2.1,\ 2\urg A_1.1,\ 2\urg A_1.3$\\
\para&$\urg A_1.2$&$\urg B_2,\ \urg A_2.2,\ 2\urg A_1.2,\ 2\urg A_1.3$\\
\midrule
&$2\urg A_1.1$&$\pcl B_2,\ (\urg A_1+\urg B_2).2,\
  \urg A_3.1,\ 3\urg A_1.1,\ 3\urg A_1.2$\\
&$2\urg A_1.2$&$\pcl B_2,\ (\urg A_1+\urg B_2).1,\
  \urg A_3.2,\ 3\urg A_1.3,\ 3\urg A_1.4$\\
\para&$2\urg A_1.3$&$\urg B_3.1,\ \urg B_3.2,\
  (\urg A_1+\urg B_2).2,\ (\urg A_1+\urg B_2).1,$\\
  &&\quad$(\urg A_1+\urg A_2).1,\ (\urg A_1+\urg A_2).2,\
  3\urg A_1.2,\ 3\urg A_1.3$\\
\para&$\urg A_2.1$&$\urg B_3.1,\ \urg A_3.1,\ (\urg A_1+\urg A_2).1$\\
\para&$\urg A_2.2$&$\urg B_3.2,\ \urg A_3.2,\ (\urg A_1+\urg A_2).2$\\
\para&$\urg B_2$&$\urg B_3.1,\ \urg B_3.2,\
  (\urg A_1+\urg B_2).1,\ (\urg A_1+\urg B_2).2$\\
\midrule
&$\urg A_3.1$&$\pcl B_3.1,\ \urg B_4.1,\
  \urg D_4.1,\ (\urg A_1+\urg A_3).1$\\
&$\urg A_3.2$&$\pcl B_3.2,\ \urg B_4.2,\
  \urg D_4.2,\ (\urg A_1+\urg A_3).2$\\
&$3\urg A_1.1$&$(\pcl A_1+\pcl B_2).2,\ (2\urg A_1+\urg B_2).1,\
  \urg D_4.1,\ 4\urg A_1.1$\\
&$3\urg A_1.2$&$\pcl B_3.1,\ (\urg A_1+\urg B_3).2,\
  (\urg A_1+\urg B_2).1,\ (2\urg A_1+\urg B_2).1,$\\
  &&\quad$(\urg A_1+\urg A_3).1,\ 4\urg A_1.2$\\
&$3\urg A_1.3$&$\pcl B_3.2,\ (\urg A_1+\urg B_3).1,\
  (\urg A_1+\urg B_2).2,\ (2\urg A_1+\urg B_2).2,$\\
  &&\quad$(\urg A_1+\urg A_3).2,\ 4\urg A_1.2$\\
&$3\urg A_1.4$&$(\pcl A_1+\pcl B_2).1,\ (2\urg A_1+\urg B_2).2,\
  \urg D_4.2,\ 4\urg A_1.3$\\
&$(\urg A_1+\urg B_2).1$&$\pcl B_3.1,\ \urg B_4.2,\
  2\urg B_2,\ (\urg A_1+\urg B_3).1,\
  (2\urg A_1+\urg B_2).2$\\
&$(\urg A_1+\urg B_2).2$&$\pcl B_3.2,\ \urg B_4.1,\
  2\urg B_2,\ (\urg A_1+\urg B_3).2,\
  (2\urg A_1+\urg B_2).1$\\
\para&$(\urg A_1+\urg A_2).1$&$\urg F_4,\ \urg B_4.1,\
  (\urg A_1+\urg B_3).1,\ 2\urg A_2,\ (\urg A_1+\urg A_3).1$\\
\para&$(\urg A_1+\urg A_2).2$&$\urg F_4,\ \urg B_4.2,\
  (\urg A_1+\urg B_3).2,\ 2\urg A_2,\ (\urg A_1+\urg A_3).2$\\
\para&$\urg B_3.1$&$\urg F_4,\ \urg B_4.1,\
  (\urg A_1+\urg B_3).1$\\
\para&$\urg B_3.2$&$\urg F_4,\ \urg B_4.2,\
  (\urg A_1+\urg B_3).2$\\
\midrule
&$4\urg A_1.1$&$\urg D_4.1,\ (2\urg A_1+\urg B_2).1$\\
&$4\urg A_1.2$&$(2\urg A_1+\urg B_2).1,\ (2\urg A_1+\urg B_2).2,\
  (\urg A_1+\urg B_3).1,\ (\urg A_1+\urg B_3).2$\\
&$4\urg A_1.3$&$\urg D_4.2,\ (2\urg A_1+\urg B_2).2$\\
&$\urg D_4.1$&$\urg B_4.1$\\
&$\urg D_4.2$&$\urg B_4.2$\\
&$(2\urg A_1+\urg B_2).1$&$2\urg B_2,\ \urg B_4.1,\
  (\urg A_1+\urg B_3).2$\\
&$(2\urg A_1+\urg B_2).2$&$2\urg B_2,\ \urg B_4.2,\
  (\urg A_1+\urg B_3).1$\\
&$2\urg B_2$&$\urg B_4.1,\ \urg B_4.2$\\
&$(\urg A_1+\urg A_3).1$&$\pcl F_4,\ \urg B_4.1,\
  (\urg A_1+\urg B_3).1$\\
&$(\urg A_1+\urg A_3).2$&$\pcl F_4,\ \urg B_4.2,\
  (\urg A_1+\urg B_3).2$\\
&$(\urg A_1+\urg B_3).1$&$\pcl F_4,\ \urg B_4.1$\\
&$(\urg A_1+\urg B_3).2$&$\pcl F_4,\ \urg B_4.2$\\
&$\urg B_4.1$&$\pcl F_4$\\
&$\urg B_4.2$&$\pcl F_4$\\
&$2\urg A_2$&$\pcl F_4$\\
\bottomrule
\end{tabular}
\end{table}

\begin{table}
\caption{Reflection subgroup classes of $G_{29}= \urg N_4$}\label{tab:N4}
\centering
\begin{tabular}{cll}
\toprule
&Class&Simple extensions\\
\toprule
\para&$\urg A_1$&$\urg B_2,\ \urg A_2,\ 2\urg A_1.1,\ 2\urg A_1.2$\\
\midrule
&$2\urg A_1.1$&$\pcl B_2,\ \urg A_1+\urg B_2,\ \urg A_3.1,\
  \urg A_3.4,\ 3\urg A_1.1,\ 3\urg A_1.2$\\
\para&$2\urg A_1.2$&$\urg B_3,\ \urg A_1+\urg B_2,\
  \urg A_1+\urg A_2,\ \urg A_3.2,\ \urg A_3.3,\ 3\urg A_1.2$\\
\para&$\urg A_2$&$\urg B_3,\ \urg D_3^{(4)},\ \urg A_1+\urg A_2,\
  \urg A_3.1,\ \urg A_3.2,\ \urg A_3.3,\ \urg A_3.4$\\
\para&$\urg B_2$&$\urg B_3,\ \urg D_3^{(4)},\
  \urg A_1+\urg B_2$\\
\midrule
&$3\urg A_1.1$&$\urg A_1+\urg B_2,\ 2\urg A_1+\urg B_2,\
  \urg D_4.1,\ 4\urg A_1.1$\\
&$3\urg A_1.2$&$\pcl B_3,\
  \urg A_1+\urg B_2,\ \urg A_1+\urg B_3,\ 2\urg A_1+\urg B_2,\
  \urg D_4.2,\ 4\urg A_1.2,\ \urg A_1+\urg A_3$\\
&$\urg A_1+\urg B_2$&$\pcl B_3,\ 2\urg B_2,\ \urg B_4,\
  \urg A_1+\urg B_3,\ 2\urg A_1+\urg B_2,\ \urg D_4^{(4)}$\\
&$\urg A_3.1$&$\pcl B_3,\ \urg D_4.1,\ \urg D_4^{(4)},\
  \urg A_1+\urg A_3$\\
&$\urg A_3.4$&$\pcl D_3^{(4)},\ \urg B_4,\ \urg D_4.1,\
  \urg D_4.2,\ \urg D_4^{(4)}$\\
\para&$\urg A_3.2$&$\urg N_4,\ \urg D_4.2,\ \urg D_4^{(4)},\ \urg A_4.1$\\
\para&$\urg A_3.3$&$\urg N_4,\ \urg D_4.2,\ \urg D_4^{(4)},\ \urg A_4.2$\\
\para&$\urg A_1+\urg A_2$&$\urg N_4,\ \urg B_4,\
  \urg A_1+\urg B_3,\ \urg A_1+\urg A_3,\ \urg A_4.1,\ \urg A_4.2$\\
\para&$\urg B_3$&$\urg N_4,\ \urg B_4,\ \urg A_1+\urg B_3$\\
\para&$\urg D_3^{(4)}$&$\urg N_4,\ \urg D_4^{(4)}$\\
\midrule
&$4\urg A_1.1$&$\urg D_4.1,\ 2\urg A_1+\urg B_2$\\
&$4\urg A_1.2$&$\urg D_4.2,\ 2\urg A_1+\urg B_2,\ \urg A_1+\urg B_3$\\
&$2\urg A_1+\urg B_2$&$\urg B_4,\ \urg A_1+\urg B_3,\
  \urg D_4^{(4)},\ 2\urg B_2$\\
&$2\urg B_2$&$\urg B_4,\ \urg D_4^{(4)}$\\
&$\urg D_4.1$&$\urg B_4,\ \urg D_4^{(4)}$\\
&$\urg D_4.2$&$\pcl N_4,\ \urg D_4^{(4)}$\\
&$\urg A_1+\urg A_3$&$\pcl N_4,\ \urg A_1+\urg B_3,\ \urg B_4$\\
&$\urg A_1+\urg B_3$&$\pcl N_4,\ \urg B_4$\\
&$\urg A_4.1$&$\pcl N_4$\\
&$\urg A_4.2$&$\pcl N_4$\\
&$\urg B_4$&$\pcl N_4$\\
&$\urg D_4^{(4)}$&$\pcl N_4$\\
\bottomrule
\end{tabular}
\end{table}

\begin{table}\small
\caption{Reflection subgroup classes of $G_{30}= \urg H_4$}\label{tab:H4}
\centering
\begin{tabular}{cll}
\toprule
&Class&Simple extensions\\
\toprule
\para&$\urg A_1$&$\urg D_2^{(5)},\ \urg A_2,\ 2\urg A_1$\\
\midrule
\para&$2\urg A_1$&$\urg H_3,\ \urg A_1+\urg D_2^{(5)},\
  \urg A_1+\urg A_2,\ \urg A_3,\ 3\urg A_1$\\
\para&$\urg A_2$&$\urg H_3,\ \urg A_1+\urg A_2,\ \urg A_3$\\
\para&$\urg D_2^{(5)}$&$\urg H_3,\ \urg A_1+\urg D_2^{(5)}$\\
\midrule
&$3\urg A_1$&$\pcl H_3,\ \urg A_1+\urg H_3,\ \urg D_4,\ 4\urg A_1$\\
\para&$\urg A_3$&$\urg H_4,\ \urg D_4,\ \urg A_4$\\
\para&$\urg A_1+\urg A_2$&$\urg H_4,\ \urg A_1+\urg H_3,\ \urg A_4,\ 2\urg A_2$\\
\para&$\urg A_1+\urg D_2^{(5)}$&$\urg H_4,\ \urg A_1+\urg H_3,\ 2\urg D_2^{(5)}$\\
\para&$\urg H_3$&$\urg H_4,\ \urg A_1+\urg H_3$\\
\midrule
&$4\urg A_1$&$\urg A_1+\urg H_3,\ \urg D_4$\\
&$\urg A_1+\urg H_3$&$\pcl H_4$\\
&$\urg D_4$&$\pcl H_4$\\
&$2\urg D_2^{(5)}$&$\pcl H_4$\\
&$\urg A_4$&$\pcl H_4$\\
&$2\urg A_2$&$\pcl H_4$\\
\bottomrule
\end{tabular}
\vskip 0.5cm
\end{table}

\begin{table}\small
\caption{Reflection subgroup classes of $G_{31}= \urg O_4$}\label{tab:O4a}
\centering
\begin{tabular}{cll}
\toprule
&Class&Simple extensions (ranks 1, 2 and 3)\\
\toprule
\para&$\urg A_1$&$\urg B_2,\ \urg A_2,\ 2\urg A_1.1,\ 2\urg A_1.2$\\
\midrule
&$2\urg A_1.2$&$\urg B_2,\ (\urg A_1+\urg B_2).1,\ \urg A_3.2,\
  3\urg A_1.1,\ 3\urg A_1.2$\\
&$\urg B_2$&$\pcl B_2^{(4)},\ \urg B_3,\ \urg D_3^{(4)},\
  (\urg A_1+\urg B_2).1,\ (\urg A_1+\urg B_2).2$\\
\para&$2\urg A_1.1$&$\urg B_3,\ (\urg A_1+\urg B_2).1,\
  (\urg A_1+\urg B_2).2,\ \urg A_1+\urg A_2,\ \urg A_3.1,\ 3\urg A_1.1$\\
\para&$\urg A_2$&$\urg B_3,\ \urg D_3^{(4)},\ \urg A_1+\urg A_2,\
  \urg A_3.1,\ \urg A_3.2$\\
\para&$\urg B_2^{(4)}$&$\urg B_3^{(4)},\ \urg A_1+\urg B_2^{(4)}$\\
\midrule
&$\urg A_3.2$&$\urg B_3,\ \urg D_3^{(4)},\ \urg B_4.1,\ \urg D_4^{(4)},\
  \urg D_4.1,\ \urg D_4.2,\ \urg A_1+\urg A_3$\\
&$3\urg A_1.1$&$\urg B_3,\ (\urg A_1+\urg B_2).1,\ (\urg A_1+\urg B_2).2,\
  \urg D_4.1,\ \urg A_1+\urg B_3,$\\
  &&\quad$\urg A_1+\urg A_3,\ (2\urg A_1+\urg B_2).1,\
  (2\urg A_1+\urg B_2).2,\ 4\urg A_1.1$\\
&$3\urg A_1.2$&$(\urg A_1+\urg B_2).1,\ \urg D_4.2,\ (2\urg A_1+\urg B_2).1,\
  4\urg A_1.2$\\
&$(\urg A_1+\urg B_2).1$&$\urg B_3,\ \urg A_1+\urg B_2^{(4)},\ \urg B_4.1,\
  2\urg B_2.1,\ 2\urg B_2.2,\ \urg D_4^{(4)},\ \urg A_1+\urg B_3,$\\
  &&\quad$(2\urg A_1+\urg B_2).1$\\
&$(\urg A_1+\urg B_2).2$&$\pcl B_3^{(4)},\ \urg A_1+\urg B_2^{(4)},\
  \urg B_4.2,\ \urg A_1+\urg D_3^{(4)},\ 2\urg B_2.2,\
  (2\urg A_1+\urg B_2).2$\\
&$\urg A_1+\urg B_2^{(4)}$&$\pcl B_3^{(4)},\ \urg B_4^{(4)},\
  \urg A_1+\urg B_3^{(4)},\ \urg B_2+\urg B_2^{(4)},\
  2\urg A_1+\urg B_2^{(4)}$\\
&$\urg B_3$&$\pcl B_3^{(4)},\ \urg N_4,\ \urg F_4,\ \urg B_4.1,\ \urg B_4.2,\
  \urg A_1+\urg B_3$\\
&$\urg D_3^{(4)}$&$\pcl B_3^{(4)},\ \urg N_4,\ \urg D_4^{(4)},\ \urg A_1+\urg D_3^{(4)}$\\
\para&$\urg A_3.1$&$\urg N_4,\ \urg B_4.2,\ \urg D_4^{(4)},\
  \urg D_4.1,\ \urg A_4.1,\ \urg A_4.2$\\
\para&$\urg A_1+\urg A_2$&$\urg N_4,\ \urg F_4,\ \urg B_4.1,\ \urg B_4.2,\
  \urg A_1+\urg B_3,\ \urg A_1+\urg D_3^{(4)},\ \urg A_1+\urg A_3,$\\
  &&\quad$2\urg A_2,\ \urg A_4.1,\ \urg A_4.2$\\
\para&$\urg B_3^{(4)}$&$\urg O_4,\ \urg B_4^{(4)},\ \urg A_1+\urg B_3^{(4)}$\\
\bottomrule
\end{tabular}
\end{table}

\begin{table}
\caption{Reflection subgroup classes of $G_{31}= \urg O_4$
(continued)}\label{tab:O4b} \centering
\begin{tabular}{cll}
\toprule
&Class&Simple extensions (rank 4)\\
\toprule
&$4\urg A_1.1$&$\urg A_1+\urg B_3,\ (2\urg A_1+\urg B_2).1,\
  (2\urg A_1+\urg B_2).2,\ \urg D_4.1$\\
&$4\urg A_1.2$&$\urg D_4.2,\ (2\urg A_1+\urg B_2).1$\\
&$(2\urg A_1+\urg B_2).1$&$\urg B_4.1,\ \urg D_4^{(4)},\ \urg A_1+\urg B_3,\
  2\urg A_1+\urg B_2^{(4)},\ 2\urg B_2.1,\ 2\urg B_2.2$\\
&$(2\urg A_1+\urg B_2).2$&$\urg A_1+\urg B_3^{(4)},\ 2\urg B_2.2,\
  \urg B_4.2,\ 2\urg A_1+\urg B_2^{(4)}$\\
&$\urg A_1+\urg A_3$&$\urg N_4,\ \urg F_4,\ \urg B_4.1,\ \urg B_4.2,\
  \urg A_1+\urg D_3^{(4)},\ \urg A_1+\urg B_3$\\
&$\urg A_1+\urg B_3$&$\urg N_4,\ \urg F_4,\
  \urg A_1+\urg B_3^{(4)},\ \urg B_4.1,\ \urg B_4.2$\\
&$2\urg B_2.1$&$\urg B_4.1,\ \urg D_4^{(4)},\ \urg B_2+\urg B_2^{(4)}$\\
&$2\urg B_2.2$&$\urg B_4.2,\ \urg B_4^{(4)},\ \urg B_2+\urg B_2^{(4)}$\\
&$2\urg A_1+\urg B_2^{(4)}$&$\urg B_4^{(4)},\ \urg A_1+\urg B_3^{(4)},\
  \urg B_2+\urg B_2^{(4)}$\\
&$\urg B_2+\urg B_2^{(4)}$&$\urg B_4^{(4)},\ 2\urg B_2^{(4)}$\\
&$\urg D_4.1$&$\urg N_4,\ \urg B_4.2,\ \urg D_4^{(4)}$\\
&$\urg D_4.2$&$\urg B_4.1,\ \urg D_4^{(4)}$\\
&$2\urg B_2^{(4)}$&$\urg B_4^{(4)}$\\
&$\urg D_4^{(4)}$&$\urg N_4,\ \urg B_4^{(4)}$\\
&$\urg B_4.1$&$\urg N_4,\ \urg B_4^{(4)},\ \urg F_4$\\
&$\urg B_4.2$&$\pcl O_4,\ \urg B_4^{(4)}$\\
&$\urg A_1+\urg D_3^{(4)}$&$\pcl O_4,\ \urg B_4^{(4)},\
  \urg A_1+\urg B_3^{(4)}$\\
&$\urg A_1+\urg B_3^{(4)}$&$\pcl O_4,\ \urg B_4^{(4)}$\\
&$\urg A_4.1$&$\pcl O_4,\ \urg N_4$\\
&$\urg A_4.2$&$\pcl O_4,\ \urg N_4$\\
&$2\urg A_2$&$\pcl O_4,\ \urg F_4$\\
&$\urg B_4^{(4)}$&$\pcl O_4$\\
&$\urg F_4$&$\pcl O_4$\\
&$\urg N_4$&$\pcl O_4$\\
\bottomrule
\end{tabular}
\vskip 0.5cm
\end{table}

\begin{table}
\caption{Reflection subgroup classes of $G_{32}= \urg L_4$}\label{tab:L4}
\centering
\begin{tabular}{cll}
\toprule
&Class&Simple extensions\\
\toprule
\para&$\urg L_1$&$\urg L_2,\ 2\urg L_1$\\
\midrule
\para&$2\urg L_1$&$\urg L_3,\ \urg L_1+\urg L_2,\ 3\urg L_1$\\
\para&$\urg L_2$&$\urg L_3,\ \urg L_1+\urg L_2$\\
\midrule
&$3\urg L_1$&$\pcl L_3,\ \urg L_1+\urg L_3,\ 4\urg L_1$\\
\para&$\urg L_3$&$\urg L_4,\ \urg L_1+\urg L_3$\\
\para&$\urg L_1+\urg L_2$&$\urg L_4,\ \urg L_1+\urg L_3,\ 2\urg L_2$\\
\midrule
&$4\urg L_1$&$\urg L_1+\urg L_3$\\
&$\urg L_1+\urg L_3$&$\pcl L_4$\\
&$2\urg L_2$&$\pcl L_4$\\
\bottomrule
\end{tabular}
\end{table}

\begin{table}
\caption{Reflection subgroup classes of $G_{33}= \urg K_5$}\label{tab:K5}
\centering
\begin{tabular}{cll}
\toprule
&Class&Simple extensions\\
\toprule
\para&$\urg A_1$&$\urg A_2,\ 2\urg A_1$\\
\midrule
\para&$2\urg A_1$&$\urg A_1+\urg A_2,\ \urg A_3,\ 3\urg A_1$\\
\para&$\urg A_2$&$\urg D_3^{(3)},\ \urg A_1+\urg A_2,\ \urg A_3$\\
\midrule
\para&$\urg A_1+\urg A_2$&$\urg D_4^{(3)},\ \urg A_1+\urg A_3,\ \urg A_4,\
  2\urg A_2$\\
\para&$\urg A_3$&$\urg D_4,\ \urg D_4^{(3)},\ \urg A_1+\urg A_3,\ \urg A_4$\\
\para&$3\urg A_1$&$\urg D_4,\ \urg A_1+\urg A_3,\ 4\urg A_1$\\
\para&$\urg D_3^{(3)}$&$\urg D_4^{(3)}$\\
\midrule
&$4\urg A_1$&$\pcl D_4,\ \urg A_1+\urg D_4,\ 5\urg A_1$\\
&$2\urg A_2$&$\pcl D_4^{(3)},\ \urg A_5$\\
\para&$\urg A_1+\urg A_3$&$\urg K_5,\ \urg A_1+\urg D_4,\ \urg A_5$\\
\para&$\urg D_4$&$\urg K_5,\ \urg A_1+\urg D_4$\\
\para&$\urg A_4$&$\urg K_5,\ \urg A_5$\\
\para&$\urg D_4^{(3)}$&$\urg K_5$\\
\midrule
&$5\urg A_1$&$\urg A_1+\urg D_4$\\
&$\urg A_1+\urg D_4$&$\pcl K_5$\\
&$\urg A_5$&$\pcl K_5$\\
\bottomrule
\end{tabular}
\vskip 0.6cm
\end{table}

\begin{table}\small
\caption{Reflection subgroup classes of $G_{34}= \urg K_6$}\label{tab:K6a}
\centering
\begin{tabular}{cll}
\toprule
&Class&Simple extensions (ranks 1 to 4)\\
\toprule
\para&$\urg A_1$&$\urg A_2,\ 2\urg A_1$\\
\midrule
\para&$2\urg A_1$&$\urg A_1+\urg A_2,\ \urg A_3,\ 3\urg A_1$\\
\para&$\urg A_2$&$\urg D_3^{(3)},\ \urg A_1+\urg A_2,\ \urg A_3$\\
\midrule
\para&$\urg A_1+\urg A_2$&$\urg D_4^{(3)},\ \urg A_1+\urg D_3^{(3)},\
  \urg A_1+\urg A_3,\ 2\urg A_1+\urg A_2,\ \urg A_4,\ 2\urg A_2.1,\
  2\urg A_2.2$\\
\para&$\urg A_3$&$\urg A_4,\ \urg D_4,\ \urg D_4^{(3)},\ \urg A_1+\urg A_3$\\
\para&$3\urg A_1$&$\urg D_4,\ \urg A_1+\urg A_3,\ 2\urg A_1+\urg A_2,\
  4\urg A_1$\\
\para&$\urg D_3^{(3)}$&$\urg D_4^{(3)},\ \urg A_1+\urg D_3^{(3)}$\\
\midrule
&$4\urg A_1$&$\pcl D_4,\ \urg A_1+\urg D_4,\ 2\urg A_1+\urg A_3,\ 5\urg A_1$\\
&$2\urg A_2.2$&$\pcl D_4^{(3)},\ \urg A_2+\urg D_3^{(3)},\
  \urg A_1+2\urg A_2,\ \urg A_5.2$\\
\para&$2\urg A_2.1$&$\urg D_5^{(3)},\ \urg A_2+\urg D_3^{(3)},\
  \urg A_2+\urg A_3,\ \urg A_5.1,\ \urg A_5.3$\\
\para&$\urg A_1+\urg A_3$&$\urg K_5,\ \urg D_5,\ \urg A_1+\urg D_4^{(3)},\
  \urg A_1+\urg D_4,\ \urg A_2+\urg A_3,\ \urg A_1+\urg A_4,$\\
  &&\quad$2\urg A_1+\urg A_3,\ \urg A_5.1,\ \urg A_5.2,\ \urg A_5.3$\\
\para&$2\urg A_1+\urg A_2$&$\urg D_5,\ \urg A_1+\urg D_4^{(3)},\
  \urg A_2+\urg A_3,\ \urg A_1+\urg A_4,\ \urg A_1+2\urg A_2,\
  2\urg A_1+\urg A_3$\\
\para&$\urg A_1+\urg D_3^{(3)}$&$\urg D_5^{(3)},\ \urg A_2+\urg D_3^{(3)},\
  \urg A_1+\urg D_4^{(3)}$\\
\para&$\urg D_4$&$\urg K_5,\ \urg D_5,\ \urg A_1+\urg D_4$\\
\para&$\urg A_4$&$\urg K_5,\ \urg D_5,\ \urg D_5^{(3)},\ \urg A_1+\urg A_4,\
  \urg A_5.1,\ \urg A_5.2,\ \urg A_5.3$\\
\para&$\urg D_4^{(3)}$&$\urg K_5,\ \urg D_5^{(3)},\ \urg A_1+\urg D_4^{(3)}$\\
\bottomrule
\end{tabular}
\end{table}

\begin{table}
\caption{Reflection subgroup classes of $G_{34}= \urg K_6$
(continued)}\label{tab:K6b} \centering
\begin{tabular}{cll}
\toprule
&Class&Simple extensions (ranks 5 and 6)\\
\toprule
&$5\urg A_1$&$\urg A_1+\urg D_4,\ 2\urg A_1+\urg D_4,\ 6\urg A_1$\\
&$\urg A_2+\urg D_3^{(3)}$&$\pcl D_5^{(3)},\ \urg D_6^{(3)},\
  2\urg D_3^{(3)},\ \urg A_2+\urg D_4^{(3)}$\\
&$\urg A_1+\urg D_4$&$\pcl K_5,\ \urg A_1+\urg K_5,\ \urg D_6,\
  2\urg A_1+\urg D_4$\\
&$2\urg A_1+\urg A_3$&$\pcl D_5,\ \urg A_1+\urg K_5,\ \urg D_6,\
  2\urg A_1+\urg D_4,\ \urg A_1+\urg A_5,\ 2\urg A_3$\\
&$\urg A_1+2\urg A_2$&$\pcl A_1+\pcl D_4^{(3)},\ \urg E_6,\ \urg A_2+\urg D_4^{(3)},\
  \urg A_1+\urg A_5,\ 3\urg A_2$\\
&$\urg A_5.2$&$\pcl K_5,\ \urg E_6,\ \urg D_6^{(3)},\ \urg A_1+\urg A_5$\\
\para&$\urg A_5.1$&$\urg K_6,\ \urg D_6,\ \urg D_6^{(3)},\ \urg A_6.1$\\
\para&$\urg A_5.3$&$\urg K_6,\ \urg D_6,\ \urg D_6^{(3)},\ \urg A_6.2$\\
\para&$\urg A_2+\urg A_3$&$\urg K_6,\ \urg D_6,\ \urg D_6^{(3)},\
  \urg A_2+\urg D_4^{(3)},\ 2\urg A_3,\ \urg A_6.1,\ \urg A_6.2$\\
\para&$\urg A_1+\urg A_4$&$\urg K_6,\ \urg E_6,\ \urg A_1+\urg K_5,\
  \urg A_1+\urg A_5,\ \urg A_6.1,\ \urg A_6.2$\\
\para&$\urg A_1+\urg D_4^{(3)}$&$\urg K_6,\ \urg A_1+\urg K_5,\ \urg D_6^{(3)},\
  \urg A_2+\urg D_4^{(3)}$\\
\para&$\urg D_5$&$\urg K_6,\ \urg E_6,\ \urg D_6$\\
\para&$\urg D_5^{(3)}$&$\urg K_6,\ \urg D_6^{(3)}$\\
\para&$\urg K_5$&$\urg K_6,\ \urg A_1+\urg K_5$\\
\midrule
&$6\urg A_1$&$2\urg A_1+\urg D_4$\\
&$2\urg A_1+\urg D_4$&$\urg A_1+\urg K_5,\ \urg D_6$\\
&$3\urg A_2$&$\urg E_6,\ \urg A_2+\urg D_4^{(3)}$\\
&$2\urg D_3^{(3)}$&$\urg D_6^{(3)}$\\
&$\urg A_1+\urg A_5$&$\pcl K_6,\ \urg A_1+\urg K_5,\ \urg E_6$\\
&$\urg A_2+\urg D_4^{(3)}$&$\pcl K_6,\ \urg D_6^{(3)}$\\
&$2\urg A_3$&$\pcl K_6,\ \urg D_6$\\
&$\urg A_6.1$&$\pcl K_6$\\
&$\urg A_6.2$&$\pcl K_6$\\
&$\urg A_1+\urg K_5$&$\pcl K_6$\\
&$\urg D_6^{(3)}$&$\pcl K_6$\\
&$\urg D_6$&$\pcl K_6$\\
&$\urg E_6$&$\pcl K_6$\\
\bottomrule
\end{tabular}
\end{table}

\begin{table}
\caption{Reflection subgroup classes of $G_{35}= \urg E_6$}\label{tab:E6}
\centering
\begin{tabular}{cll}
\toprule
&Class&Simple extensions\\
\toprule
\para&$\urg A_1$&$\urg A_2,\ 2\urg A_1$\\
\midrule
\para&$2\urg A_1$&$\urg A_1+\urg A_2,\ \urg A_3,\ 3\urg A_1$\\
\para&$\urg A_2$&$\urg A_1+\urg A_2,\ \urg A_3$\\
\midrule
\para&$\urg A_1+\urg A_2$&$2\urg A_1+\urg A_2,\ \urg A_1+\urg A_3,\
  \urg A_4,\ 2\urg A_2$\\
\para&$\urg A_3$&$\urg D_4,\ \urg A_1+\urg A_3,\ \urg A_4$\\
\para&$3\urg A_1$&$\urg D_4,\ \urg A_1+\urg A_3,\
  2\urg A_1+\urg A_2,\ 4\urg A_1$\\
\midrule
&$4\urg A_1$&$\pcl D_4,\ 2\urg A_1+\urg A_3$\\
\para&$2\urg A_1+\urg A_2$&$\urg D_5,\ \urg A_1+\urg A_4,\ 2\urg A_1+\urg A_3,\
  \urg A_1+2\urg A_2$\\
\para&$\urg A_1+\urg A_3$&$\urg D_5,\ \urg A_1+\urg A_4,\ 2\urg A_1+\urg A_3,\
  \urg A_5$\\
\para&$\urg A_4$&$\urg D_5,\ \urg A_1+\urg A_4,\ \urg A_5$\\
\para&$2\urg A_2$&$\urg A_1+2\urg A_2,\ \urg A_5$\\
\para&$\urg D_4$&$\urg D_5$\\
\midrule
&$2\urg A_1+\urg A_3$&$\pcl D_5,\ \urg A_1+\urg A_5$\\
\para&$\urg A_1+2\urg A_2$&$\urg E_6,\ \urg A_1+\urg A_5,\ 3\urg A_2$\\
\para&$\urg A_1+\urg A_4$&$\urg E_6,\ \urg A_1+\urg A_5$\\
\para&$\urg A_5$&$\urg E_6,\ \urg A_1+\urg A_5$\\
\para&$\urg D_5$&$\urg E_6$\\
\midrule
&$\urg A_1+\urg A_5$&$\pcl E_6$\\
&$3\urg A_2$&$\pcl E_6$\\
\bottomrule
\end{tabular}
\end{table}

\begin{table}
\caption{Reflection subgroup classes of $G_{36}= \urg E_7$}\label{tab:E7a}
\centering
\begin{tabular}{cll}
\toprule
&Class&Simple extensions (ranks 1 to 5)\\
\toprule
\para&$\urg A_1$&$2\urg A_1,\ \urg A_2$\\
\midrule
\para&$2\urg A_1$&$\urg A_1+\urg A_2,\ \urg A_3,\ 3\urg A_1.1,\ 3\urg A_1.2$\\
\para&$\urg A_2$&$\urg A_3,\ \urg A_1+\urg A_2$\\
\midrule
\para&$\urg A_1+\urg A_2$&$(\urg A_1+\urg A_3).1,\ (\urg A_1+\urg A_3).2,\
  2\urg A_1+\urg A_2,\ \urg A_4,\ 2\urg A_2$\\
\para&$\urg A_3$&$\urg D_4,\ (\urg A_1+\urg A_3).1,\ (\urg A_1+\urg A_3).2,\
  \urg A_4$\\
\para&$3\urg A_1.1$&$\urg D_4,\ 2\urg A_1+\urg A_2,\
  (\urg A_1+\urg A_3).1,\ 4\urg A_1.1,\ 4\urg A_1.2$\\
\para&$3\urg A_1.2$&$(\urg A_1+\urg A_3).2,\ 4\urg A_1.2$\\
\midrule
&$4\urg A_1.1$&$\pcl D_4,\ (2\urg A_1+\urg A_3).1,\ 5\urg A_1$\\
\para&$4\urg A_1.2$&$\urg A_1+\urg D_4,\ (2\urg A_1+\urg A_3).2,\
  3\urg A_1+\urg A_2,\ 5\urg A_1$\\
\para&$(\urg A_1+\urg A_3).1$&$\urg D_5,\ \urg A_1+\urg D_4,\
  \urg A_1+\urg A_4,\ \urg A_2+\urg A_3,$\\
  &&\quad$(2\urg A_1+\urg A_3).1,\ (2\urg A_1+\urg A_3).2,\ \urg A_5.1$\\
\para&$(\urg A_1+\urg A_3).2$&$\urg A_1+\urg D_4,\
  (2\urg A_1+\urg A_3).2,\ \urg A_5.2$\\
\para&$2\urg A_1+\urg A_2$&$\urg D_5,\ \urg A_1+\urg A_4,\
  \urg A_1+2\urg A_2,\ \urg A_2+\urg A_3,\ 3\urg A_1+\urg A_2,$\\
  &&\quad$(2\urg A_1+\urg A_3).1,\ (2\urg A_1+\urg A_3).2$\\
\para&$\urg A_4$&$\urg D_5,\ \urg A_1+\urg A_4,\ \urg A_5.1,\ \urg A_5.2$\\
\para&$2\urg A_2$&$\urg A_2+\urg A_3,\ \urg A_1+2\urg A_2,\ \urg A_5.1,\ \urg A_5.2$\\
\para&$\urg D_4$&$\urg D_5,\ \urg A_1+\urg D_4$\\
\midrule
&$5\urg A_1$&$\pcl A_1+\pcl D_4,\ 2\urg A_1+\urg D_4,\
  3\urg A_1+\urg A_3,\ 6\urg A_1$\\
&$(2\urg A_1+\urg A_3).1$&$\pcl D_5,\ 2\urg A_1+\urg D_4,\
  3\urg A_1+\urg A_3,\ (\urg A_1+\urg A_5).1,\ 2\urg A_3$\\
\para&$(2\urg A_1+\urg A_3).2$&$\urg D_6,\ \urg A_1+\urg D_5,\
  2\urg A_1+\urg D_4,\ (\urg A_1+\urg A_5).2,$\\
  &&\quad$\urg A_1+\urg A_2+\urg A_3,\ 3\urg A_1+\urg A_3$\\
\para&$\urg A_1+\urg D_4$&$\urg D_6,\ \urg A_1+\urg D_5,\
  2\urg A_1+\urg D_4$\\
\para&$\urg A_1+\urg A_4$&$\urg E_6,\ \urg A_2+\urg A_4,\ \urg A_1+\urg D_5,\
  \urg A_6,\ (\urg A_1+\urg A_5).1,\ (\urg A_1+\urg A_5).2$\\
\para&$\urg A_2+\urg A_3$&$\urg D_6,\ \urg A_6,\ \urg A_1+\urg A_2+\urg A_3,\
  \urg A_2+\urg A_4,\ 2\urg A_3$\\
\para&$\urg A_1+2\urg A_2$&$\urg E_6,\ \urg A_1+\urg A_2+\urg A_3,\
  (\urg A_1+\urg A_5).1,\ (\urg A_1+\urg A_5).2,\ \urg A_2+\urg A_4,\ 3\urg A_2$\\
\para&$3\urg A_1+\urg A_2$&$\urg A_1+\urg D_5,\ 3\urg A_1+\urg A_3,\
  \urg A_1+\urg A_2+\urg A_3$\\
\para&$\urg A_5.1$&$\urg E_6,\ \urg A_6,\ \urg D_6,\ (\urg A_1+\urg A_5).1$\\
\para&$\urg A_5.2$&$\urg D_6,\ (\urg A_1+\urg A_5).2$\\
\para&$\urg D_5$&$\urg E_6,\ \urg D_6,\ \urg A_1+\urg D_5$\\
\bottomrule
\end{tabular}
\end{table}

\begin{table}
\caption{Reflection subgroup classes of $G_{36}= \urg E_7$
(continued)}\label{tab:E7b} \centering
\begin{tabular}{cll}
\toprule
&Class&Simple extensions (ranks 6 and 7)\\
\toprule
&$6\urg A_1$&$2\urg A_1+\urg D_4,\ 3\urg A_1+\urg D_4,\ 7\urg A_1$\\
&$2\urg A_1+\urg D_4$&$\pcl D_6,\ \urg A_1+\urg D_6,\
  3\urg A_1+\urg D_4$\\
&$2\urg A_3$&$\pcl D_6,\ \urg A_7,\ \urg A_1+2\urg A_3$\\
&$3\urg A_1+\urg A_3$&$\pcl A_1+\pcl D_5,\ \urg A_1+\urg D_6,\
  \urg A_1+2\urg A_3,\ 3\urg A_1+\urg D_4$\\
&$3\urg A_2$&$\pcl E_6,\ \urg A_2+\urg A_5$\\
&$(\urg A_1+\urg A_5).1$&$\pcl E_6,\ \urg A_1+\urg D_6,\ \urg A_7$\\
\para&$(\urg A_1+\urg A_5).2$&$\urg E_7,\ \urg A_2+\urg A_5,\ \urg A_1+\urg D_6$\\
\para&$\urg A_1+\urg D_5$&$\urg E_7,\ \urg A_1+\urg D_6$\\
\para&$\urg A_2+\urg A_4$&$\urg E_7,\ \urg A_2+\urg A_5,\ \urg A_7$\\
\para&$\urg A_1+\urg A_2+\urg A_3$&$\urg E_7,\ \urg A_1+\urg D_6,\
  \urg A_2+\urg A_5,\ \urg A_1+2\urg A_3$\\
\para&$\urg A_6$&$\urg E_7,\ \urg A_7$\\
\para&$\urg D_6$&$\urg E_7,\ \urg A_1+\urg D_6$\\
\para&$\urg E_6$&$\urg E_7$\\
\midrule
&$7\urg A_1$&$3\urg A_1+\urg D_4$\\
&$3\urg A_1+\urg D_4$&$\urg A_1+\urg D_6$\\
&$\urg A_1+2\urg A_3$&$\pcl E_7,\ \urg A_1+\urg D_6$\\
&$\urg A_1+\urg D_6$&$\pcl E_7$\\
&$\urg A_2+\urg A_5$&$\pcl E_7$\\
&$\urg A_7$&$\pcl E_7$\\
\bottomrule
\end{tabular}
\vskip 0.5cm
\end{table}

\begin{table}
\caption{Reflection subgroup classes of $G_{37}= \urg E_8$}\label{tab:E8a}
\centering
\begin{tabular}{cll}
\toprule
&Class&Simple extensions (ranks 1 to 4)\\
\toprule
\para&$\urg A_1$&$2\urg A_1,\ \urg A_2$\\
\midrule
\para&$2\urg A_1$&$\urg A_1+\urg A_2,\ 3\urg A_1,\ \urg A_3$\\
\para&$\urg A_2$&$\urg A_1+\urg A_2,\ \urg A_3$\\
\midrule
\para&$\urg A_1+\urg A_2$&$2\urg A_1+\urg A_2,\ 2\urg A_2,\ \urg A_4,\ \urg A_1+\urg A_3$\\
\para&$\urg A_3$&$\urg D_4,\ \urg A_4,\ \urg A_1+\urg A_3$\\
\para&$3\urg A_1$&$\urg D_4,\ 2\urg A_1+\urg A_2,\ \urg A_1+\urg A_3,\
  4\urg A_1.1,\ 4\urg A_1.2$\\
\midrule
&$4\urg A_1.1$&$\pcl D_4,\ (2\urg A_1+\urg A_3).1,\ 5\urg A_1$\\
\para&$4\urg A_1.2$&$\urg A_1+\urg D_4,\ (2\urg A_1+\urg A_3).2,\ 3\urg A_1+\urg A_2,\
  5\urg A_1$\\
\para&$2\urg A_1+\urg A_2$&$\urg D_5,\ \urg A_1+\urg A_4,\ \urg A_1+2\urg A_2,\
  \urg A_2+\urg A_3,\ 3\urg A_1+\urg A_2,$\\
  &&\quad$(2\urg A_1+\urg A_3).1,\ (2\urg A_1+\urg A_3).2$\\
\para&$2\urg A_2$&$\urg A_2+\urg A_3,\ \urg A_5,\ \urg A_1+2\urg A_2$\\
\para&$\urg A_4$&$\urg D_5,\ \urg A_1+\urg A_4,\ \urg A_5$\\
\para&$\urg A_1+\urg A_3$&$\urg D_5,\ \urg A_1+\urg D_4,\
  \urg A_1+\urg A_4,\ \urg A_2+\urg A_3,\ \urg A_5,$\\
  &&\quad$(2\urg A_1+\urg A_3).1,\ (2\urg A_1+\urg A_3).2$\\
\para&$\urg D_4$&$\urg D_5,\ \urg A_1+\urg D_4$\\
\bottomrule
\end{tabular}
\end{table}

\clearpage

\begin{table}
\caption{Reflection subgroup classes of $G_{37}= \urg E_8$
(continued)}\label{tab:E8b} \centering
\begin{tabular}{cll}
\toprule
&Class&Simple extensions (ranks 5 and 6)\\
\toprule
&$5\urg A_1$&$\pcl A_1+\pcl D_4,\ 3\urg A_1+\urg A_3,\ 2\urg A_1+\urg D_4,\
  4\urg A_1+\urg A_2,\ 6\urg A_1$\\
&$(2\urg A_1+\urg A_3).1$&$\pcl D_5,\ (\urg A_1+\urg A_5).1,\
  2\urg A_3.1,\ 2\urg A_1+\urg D_4,\ 3\urg A_1+\urg A_3$\\
\para&$(2\urg A_1+\urg A_3).2$&$\urg D_6,\ (\urg A_1+\urg A_5).2,\
  2\urg A_3.2,\ 2\urg A_1+\urg D_4,\ 3\urg A_1+\urg A_3,$\\
  &&\quad$\urg A_1+\urg D_5,\ 2\urg A_1+\urg A_4,\ \urg A_1+\urg A_2+\urg A_3$\\
\para&$\urg A_1+\urg A_4$&$\urg E_6,\ \urg A_1+\urg D_5,\
  (\urg A_1+\urg A_5).1,\ (\urg A_1+\urg A_5).2,\ \urg A_2+\urg A_4,$\\
  &&\quad$2\urg A_1+\urg A_4,\ \urg A_6$\\
\para&$\urg A_1+2\urg A_2$&$\urg E_6,\ (\urg A_1+\urg A_5).1,\
  (\urg A_1+\urg A_5).2,\ \urg A_2+\urg A_4,\ \urg A_1+\urg A_2+\urg A_3,$\\
  &&\quad$2\urg A_1+2\urg A_2,\ 3\urg A_2$\\
\para&$\urg A_2+\urg A_3$&$\urg D_6,\ \urg A_2+\urg D_4,\ \urg A_6,\
  \urg A_2+\urg A_4,\ 2\urg A_3.1,\ 2\urg A_3.2,\ \urg A_1+\urg A_2+\urg A_3$\\
\para&$3\urg A_1+\urg A_2$&$\urg A_1+\urg D_5,\ \urg A_2+\urg D_4,\
  \urg A_1+\urg A_2+\urg A_3,\ 2\urg A_1+\urg A_4,\ 3\urg A_1+\urg A_3,$\\
  &&\quad$2\urg A_1+2\urg A_2,\ 4\urg A_1+\urg A_2$\\
\para&$\urg D_5$&$\urg E_6,\ \urg D_6,\ \urg A_1+\urg D_5$\\
\para&$\urg A_5$&$\urg E_6,\ \urg D_6,\ \urg A_6,\ (\urg A_1+\urg A_5).1,\
  (\urg A_1+\urg A_5).2$\\
\para&$\urg A_1+\urg D_4$&$\urg D_6,\ \urg A_2+\urg D_4,\
  \urg A_1+\urg D_5,\ 2\urg A_1+\urg D_4$\\
\midrule
&$6\urg A_1$&$2\urg A_1+\urg D_4,\ 3\urg A_1+\urg D_4,\ 4\urg A_1+\urg A_3,\
  7\urg A_1$\\
&$2\urg A_1+\urg D_4$&$\pcl D_6,\ \urg A_1+\urg D_6,\
  \urg A_3+\urg D_4,\ 2\urg A_1+\urg D_5,\ 3\urg A_1+\urg D_4$\\
&$3\urg A_1+\urg A_3$&$\pcl A_1+\pcl D_5,\ 2\urg A_1+\urg A_5,\
  \urg A_1+\urg D_6,\ 2\urg A_1+\urg D_5,\ \urg A_3+\urg D_4,$\\
  &&\quad$\urg A_1+2\urg A_3,\ 3\urg A_1+\urg D_4,\
  2\urg A_1+\urg A_2+\urg A_3,\ 4\urg A_1+\urg A_3$\\
&$4\urg A_1+\urg A_2$&$\pcl A_2+\pcl D_4,\ 2\urg A_1+\urg D_5,\
  2\urg A_1+\urg A_2+\urg A_3,\ 4\urg A_1+\urg A_3$\\
&$3\urg A_2$&$\pcl E_6,\ \urg A_2+\urg A_5,\ \urg A_1+3\urg A_2$\\
&$2\urg A_3.1$&$\pcl D_6,\ \urg A_7.1,\ \urg A_3+\urg D_4,\
  \urg A_1+2\urg A_3$\\
&$(\urg A_1+\urg A_5).1$&$\pcl E_6,\ \urg A_7.1,\ \urg A_1+\urg D_6,\
  2\urg A_1+\urg A_5$\\
\para&$2\urg A_3.2$&$\urg A_7.2,\ \urg D_7,\ \urg A_3+\urg A_4,\
  \urg A_3+\urg D_4$\\
\para&$(\urg A_1+\urg A_5).2$&$\urg E_7,\ \urg A_7.2,\ \urg A_1+\urg D_6,\
  2\urg A_1+\urg A_5,\ \urg A_1+\urg E_6,\ \urg A_1+\urg A_6,\ \urg A_2+\urg A_5$\\
\para&$\urg E_6$&$\urg E_7,\ \urg A_1+\urg E_6$\\
\para&$\urg D_6$&$\urg E_7,\ \urg D_7,\ \urg A_1+\urg D_6$\\
\para&$\urg A_6$&$\urg E_7,\ \urg D_7,\ \urg A_7.1,\ \urg A_7.2,\ \urg A_1+\urg A_6$\\
\para&$\urg A_2+\urg A_4$&$\urg E_7,\ \urg A_7.1,\ \urg A_7.2,\ \urg A_3+\urg A_4,\
  \urg A_2+\urg D_5,\ \urg A_2+\urg A_5,\ \urg A_1+\urg A_2+\urg A_4$\\
\para&$\urg A_1+\urg D_5$&$\urg E_7,\ \urg A_1+\urg D_6,\ \urg D_7,\
  \urg A_1+\urg E_6,\ 2\urg A_1+\urg D_5,\ \urg A_2+\urg D_5$\\
\para&$\urg A_1+\urg A_2+\urg A_3$&$\urg E_7,\ \urg A_1+\urg D_6,\
  \urg A_2+\urg D_5,\ \urg A_3+\urg A_4,\ \urg A_1+\urg A_6,\
  2\urg A_1+\urg A_2+\urg A_3,$\\
  &&\quad$\urg A_1+\urg A_2+\urg A_4,\ \urg A_2+\urg A_5,\ \urg A_1+2\urg A_3$\\
\para&$2\urg A_1+\urg A_4$&$\urg D_7,\ 2\urg A_1+\urg A_5,\
  \urg A_1+\urg A_6,\ \urg A_3+\urg A_4,\ 2\urg A_1+\urg D_5,$\\
  &&\quad$\urg A_1+\urg E_6,\ \urg A_1+\urg A_2+\urg A_4$\\
\para&$2\urg A_1+2\urg A_2$&$2\urg A_1+\urg A_5,\ \urg A_2+\urg D_5,\
  \urg A_1+\urg A_2+\urg A_4,\ \urg A_1+\urg E_6,\ \urg A_1+3\urg A_2,$\\
  &&\quad$2\urg A_1+\urg A_2+\urg A_3$\\
\para&$\urg A_2+\urg D_4$&$\urg D_7,\ \urg A_2+\urg D_5,\
  \urg A_3+\urg D_4$\\
\bottomrule
\end{tabular}
\end{table}

\begin{table}
\caption{Reflection subgroup classes of $G_{37}= \urg E_8$
(continued)}\label{tab:E8c} \centering
\begin{tabular}{cll}
\toprule
&Class&Simple extensions (ranks 7 and 8)\\
\toprule
&$7\urg A_1$&$3\urg A_1+\urg D_4,\ 4\urg A_1+\urg D_4,\ 8\urg A_1$\\
&$4\urg A_1+\urg A_3$&$2\urg A_1+\urg D_5,\ \urg A_3+\urg D_4,\
  2\urg A_1+\urg D_6,\ 2\urg A_1+2\urg A_3,\ 4\urg A_1+\urg D_4$\\
&$3\urg A_1+\urg D_4$&$\pcl A_1+\pcl D_6,\ 2\urg A_1+\urg D_6,\
  2\urg D_4,\ 4\urg A_1+\urg D_4$\\
&$2\urg A_1+\urg D_5$&$\pcl D_7,\ \urg A_1+\urg E_7,\
  2\urg A_1+\urg D_6,\ \urg A_3+\urg D_5$\\
&$\urg A_3+\urg D_4$&$\pcl D_7,\ \urg D_8,\ \urg A_3+\urg D_5,\
  2\urg D_4$\\
&$2\urg A_1+\urg A_2+\urg A_3$&$\pcl A_2+\pcl D_5,\
  2\urg A_1+\urg D_6,\ \urg A_1+\urg E_7,\
  \urg A_3+\urg D_5,\ \urg A_1+\urg A_2+\urg A_5,$\\
  &&\quad$2\urg A_1+2\urg A_3$\\
&$\urg A_1+3\urg A_2$&$\pcl A_1+\pcl E_6,\ \urg A_1+\urg A_2+\urg A_5,\
  \urg A_2+\urg E_6,\ 4\urg A_2$\\
&$2\urg A_1+\urg A_5$&$\pcl A_1+\pcl E_6,\ \urg D_8,\ \urg A_1+\urg A_7,\
  2\urg A_1+\urg D_6,\ \urg A_1+\urg E_7,\ \urg A_1+\urg A_2+\urg A_5$\\
&$\urg A_1+2\urg A_3$&$\pcl E_7,\ \urg A_1+\urg D_6,\ \urg A_3+\urg D_5,\
  \urg A_1+\urg A_7,\ 2\urg A_1+2\urg A_3$\\
&$\urg A_2+\urg A_5$&$\pcl E_7,\ \urg A_8,\ \urg A_2+\urg E_6,\
  \urg A_1+\urg A_2+\urg A_5$\\
&$\urg A_1+\urg D_6$&$\pcl E_7,\ \urg D_8,\ \urg A_1+\urg E_7,\
  2\urg A_1+\urg D_6$\\
&$\urg A_7.1$&$\pcl E_7,\ \urg D_8,\ \urg A_1+\urg A_7$\\
\para&$\urg A_7.2$&$\urg E_8,\ \urg D_8,\ \urg A_8$\\
\para&$\urg D_7$&$\urg E_8,\ \urg D_8$\\
\para&$\urg E_7$&$\urg E_8,\ \urg A_1+\urg E_7$\\
\para&$\urg A_1+\urg A_6$&$\urg E_8,\ \urg A_1+\urg A_7,\ \urg A_8,\ \urg A_1+\urg E_7$\\
\para&$\urg A_1+\urg E_6$&$\urg E_8,\ \urg A_1+\urg E_7,\ \urg A_2+\urg E_6$\\
\para&$\urg A_3+\urg A_4$&$\urg E_8,\ \urg D_8,\ \urg A_8,\
  \urg A_3+\urg D_5,\ 2\urg A_4$\\
\para&$\urg A_2+\urg D_5$&$\urg E_8,\ \urg D_8,\ \urg A_2+\urg E_6,\
  \urg A_3+\urg D_5$\\
\para&$\urg A_1+\urg A_2+\urg A_4$&$\urg E_8,\ \urg A_1+\urg E_7,\
  \urg A_2+\urg E_6,\ \urg A_1+\urg A_7,\ \urg A_1+\urg A_2+\urg A_5,\ 2\urg A_4$\\
\midrule
&$2\urg A_1+2\urg A_3$&$2\urg A_1+\urg D_6,\ \urg A_1+\urg E_7,\
  \urg A_3+\urg D_5$\\
&$8\urg A_1$&$4\urg A_1+\urg D_4$\\
&$4\urg A_1+\urg D_4$&$2\urg D_4,\ 2\urg A_1+\urg D_6$\\
&$2\urg A_1+\urg D_6$&$\urg D_8,\ \urg A_1+\urg E_7$\\
&$4\urg A_2$&$\urg A_2+\urg E_6$\\
&$2\urg D_4$&$\urg D_8$\\
&$\urg A_1+\urg A_2+\urg A_5$&$\pcl E_8,\ \urg A_1+\urg E_7,\ \urg A_2+\urg E_6$\\
&$\urg A_3+\urg D_5$&$\pcl E_8,\ \urg D_8$\\
&$\urg A_1+\urg A_7$&$\pcl E_8,\ \urg A_1+\urg E_7$\\
&$2\urg A_4$&$\pcl E_8$\\
&$\urg A_2+\urg E_6$&$\pcl E_8$\\
&$\urg A_8$&$\pcl E_8$\\
&$\urg A_1+\urg E_7$&$\pcl E_8$\\
&$\urg D_8$&$\pcl E_8$\\
\bottomrule
\end{tabular}
\end{table}

\clearpage
\bibliographystyle{abbrv}
\def\noopsort#1{}\def\cprime{$'$}

\end{document}